\documentclass[10pt, oneside, a4paper]{article}
\usepackage{geometry}
\usepackage[english]{babel}

\usepackage[T1]{fontenc}
\usepackage[latin1]{inputenc}

\usepackage{lmodern}

\usepackage{amssymb,amsmath,amscd} 
\usepackage{mathrsfs}
\usepackage{stmaryrd}

\usepackage{ntheorem}

{\theoremstyle{plain}
\theoremseparator{ --}
\newtheorem{theo}{Theorem}[section]}

{\theoremstyle{plain}
\theoremseparator{ --}
\newtheorem{coro}[theo]{Corollary}}

{\theoremseparator{ --}
\newtheorem{lemm}[theo]{Lemma}}

{\theoremseparator{ --}
\newtheorem{splemm}{Lemma}}

{\theoremseparator{ --}
\newtheorem{prop}[theo]{Proposition}}

{\theoremseparator{ --}
}

{\theoremseparator{ --}
}

{\theoremstyle{plain}
\theorembodyfont{\normalfont}
}

{\theoremstyle{plain}
\theorembodyfont{\normalfont}
\newtheorem{rema}[theo]{Remark}}

{\theoremheaderfont{\itshape}
\theorembodyfont{\normalfont}
\theoremseparator{: }
\newtheorem*{proof}{Proof}}

\newcommand{\ProofEnd}{\hfill $\Box$}

\usepackage{graphicx}
\usepackage{tabularx}
\usepackage[table]{xcolor}	
\usepackage{array}
\usepackage{multirow}
\usepackage{todonotes}
\usepackage{enumerate}
\usepackage{multicol}

\usepackage{fancyhdr}

\usepackage{url}

\DeclareMathOperator{\Reg}{Reg}
\DeclareMathOperator{\ord}{ord}
\DeclareMathOperator{\Gal}{Gal}

\DeclareMathOperator{\DEGRE}{deg }
\renewcommand{\deg}{\DEGRE}
\newcommand{\tam}{\tau}

\newcommand{\R}{\ensuremath{\mathbb{R}}}
\newcommand{\Q}{\ensuremath{\mathbb{Q}}}

\newcommand{\Z}{\ensuremath{\mathbb{Z}}}
\renewcommand{\P}{\ensuremath{\mathbb{P}}}
\newcommand{\F}{\ensuremath{\mathbb{F}}}
\renewcommand{\H}{\ensuremath{\mathrm{H}}}

\newcommand{\dd}{\ensuremath{\mathrm{d}}}

\newcommand{\Qbar}{\ensuremath{\bar{\mathbb{Q}}}}

\renewcommand{\bar}[1]{\ensuremath{\overline{#1}}}

\newcommand{\hhat}[1]{\ensuremath{\widehat{#1}}}

\renewcommand{\tt}{\mathbf{t}}

 \newcommand{\ind}{1\mkern-4.1mu\mathrm{l}}

\newcommand{\ie}{\textit{i.e.}{}}
\newcommand{\cf}{\textit{cf.}{}}

\input cyracc.def
\DeclareFontFamily{U}{russian}{}
\DeclareFontShape{U}{russian}{m}{n}
	{ <5><6> wncyr5
	<7><8><9> wncyr7
	<10><10.95><12><14.4><17.28><20.74><24.88> wncyr10 }{}
\DeclareSymbolFont{Russian}{U}{russian}{m}{n}
\DeclareSymbolFontAlphabet{\mathcyr}{Russian}
\makeatletter
\let\@math@cyr\mathcyr
\renewcommand{\mathcyr}[1]{\@math@cyr{\cyracc #1}}
\makeatother
\newcommand{\sha}{\mathcyr{SH}}

\newcommand{\sbgp}[2]{{\langle#1\rangle_{#2}}}

\newcommand{\cond}{\mathcal{N}}
\newcommand{\type}[1]{\mathbf{#1}}

\definecolor{myblue}{rgb}{0.2,0.6,1}
\definecolor{myblue2}{rgb}{0,0.4,0.8}
\definecolor{noir}{rgb}{0,0,0}
\definecolor{boxescol}{rgb}{0.25,0.25,0.25}

\usepackage{bm}

\newcommand{\intent}[1]{\llbracket #1\rrbracket}
\newcommand{\partfrac}[1]{\left\{#1\right\}}
\newcommand{\partint}[1]{\left\lfloor#1\right\rfloor}

\newcommand{\tors}{_{\mathrm{tors}}}

\newcommand{\Ja}{\mathbf{J}} 
\newcommand{\Ecal}{\mathcal{E}}

\newcommand{\Mm}{\mathbf{m}}
\newcommand{\Mn}{\mathbf{n}}

\newcommand{\norm}{\mathbf{N}}

\newcommand{\gP}{\mathfrak{P}}
\newcommand{\gp}{\mathfrak{p}}

\newcommand{\BS}{\mathfrak{Bs}}

\renewcommand{\epsilon}{\varepsilon}
\newcommand{\IND}{\ind} 

\usepackage[colorlinks=true,
citecolor=red!50, 
linkcolor=noir!60,
urlcolor=blue!50, pagebackref
]{hyperref}

\title{Analogue of the  Brauer-Siegel theorem for Legendre elliptic curves}
\author{Richard {Griffon}\footnote{ 
E-mail: \url{r.m.m.griffon@math.leidenuniv.nl}
}  
$\ $ (Universiteit Leiden)
}
\date{}

\renewcommand{\O}{\mathcal{O}}

\begin{document}
\pagestyle{fancy}

\maketitle

\hfill\rule{5cm}{0.5pt}\hfill\phantom{.}

\paragraph{Abstract --}We prove an analogue of the Brauer--Siegel theorem for the Legendre elliptic curves over $\F_q(t)$.  
Namely, denoting by $E_d$  the elliptic curve with model $y^2=x(x+1)(x+t^d)$ over $K=\F_q(t)$, 
 we show that, 
for $d\to\infty$ ranging over the integers coprime with $q$, one has
\[ \log\left( |\sha(E_d/K)|\cdot \Reg(E_d/K)\right) \sim \log H(E_d/K) \sim \frac{\log q}{2} \cdot d.\]
Here, $H(E_d/K)$ denotes the exponential differential height of $E_d$, $\sha(E_d/K)$ its Tate--Shafarevich group (which is known to be finite), and $\Reg(E_d/K)$ its Néron--Tate regulator.

\hfill\rule{5cm}{0.5pt}\hfill\phantom{.}

\medskip
\noindent{%

\noindent{
{\it 2010 Math. Subj. Classification:} 
11G05, 
11G40, 
11F67, 
14G10, 
11M99, 
11R47. 

\section*{Introduction}
\pdfbookmark[0]{Introduction}{Introduction} 
\addcontentsline{toc}{section}{Introduction}

\setcounter{section}{1}

\vspace{-3mm}
\indent{}

 The Brauer--Siegel theorem describes the asymptotic behaviour of the product of the class number by the regulator of units  in sequences of number fields. More precisely, when $k$ runs through a sequence of number fields whose degrees over $\Q$ are bounded, and such that the absolute values $\Delta_k$ of their discriminants tend to $+\infty$, 
 then one has the asymptotic estimate
  \begin{equation}\label{eq.i.classBS}
 \log\big( |\mathrm{C}l(k)|\cdot \Reg(k) \big) \sim \log \sqrt{\Delta_k}
\qquad (\text{as } \Delta_k\to\infty), 
 \end{equation}
 where $\mathrm{C}l(k)$ denotes the class-group of $k$, and $\Reg(k)$ its regulator of units. 
 
 \medskip

 In their recent paper \cite{HP15}, Hindry and Pacheco proposed to study an analogue of \eqref{eq.i.classBS} for elliptic curves $E$ over  $K=\F_q(t)$, where $\F_q$ is a given finite field. 
The analogy is as follows: one replaces $\sqrt{\Delta_k}$ by the exponential height of $E$, the class number $|\mathrm{C}l(k)|$ by the order of the Tate--Shafarevich group $\sha(E/K)$ (assuming it is finite), and the regulator of units $\Reg(k)$ by the Néron-Tate regulator $\Reg(E/K)$. 
They were thus led to introduce the \emph{Brauer--Siegel ratio} of $E/K$: 
 \[ \BS(E/K) := \frac{\log\big(|\sha(E/K)|\cdot \Reg(E/K)\big)}{\log H(E/K)}, \] 
and to investigate its 
asymptotic behaviour 
as $E$ runs through sequences of elliptic curves over $K$ whose heights tend to $+\infty$. Assuming the finiteness of Tate--Shafarevich groups, they prove that
\[ 0 \leq \liminf_{H(E/K)\to\infty} \BS(E/K) \leq \limsup_{H(E/K)\to\infty} \BS(E/K) =1. 
\]

Should a perfect analogue of \eqref{eq.i.classBS} for elliptic curves hold, one would certainly expect that 
\[\lim_{H(E/K)\to\infty} \BS(E/K) =1.\]
However, not only is the proof of such an asymptotic relation out of reach at the moment, but one can reasonably doubt that it should hold in general. Indeed, Hindry and Pacheco discuss heuristics suggesting the existence of infinite sequences $\{E_n\}_{n\geq 1}$ of elliptic curves for which $\lim_{n\to\infty} \BS(E_n/K) =0$. 

\medskip

The goal of this article is to exhibit a sequence $\{E_d\}_{d}$ of elliptic curves over $K$ that \emph{does} satisfy unconditionally a complete analogue of the classical Brauer-Siegel theorem \eqref{eq.i.classBS}. 
Indeed, our main theorem~is:

 \begin{theo}\label{theo.BS.Leg}
 Let $\F_q$ be a finite field of odd characteristic, and $K=\F_q(t)$. For any integer $d\geq 2$ coprime with $q$, consider the Legendre elliptic curve $E_d/K$ defined by the affine Weierstrass model:
\begin{equation}\notag 
E_d:\quad y^2=x(x+1)(x+t^d).
\end{equation}
Then the Tate--Shafarevich group $\sha(E_d/K)$ is finite and,
as $d\to\infty$, one has the asymptotic estimate:
\begin{equation} \label{eq.BS.Leg}
\log\big( |\sha(E_d/K)|\cdot \Reg(E_d/K)\big) \sim \log H(E_d/K),
\end{equation}
where $\Reg(E_d/K)$ denotes the Néron-Tate regulator of $E_d(K)$, and $H(E_d/K)$ is the (exponential) differential height of $E_d/K$ (see definitions below).
\end{theo}
 
This theorem can be restated as:
\[\forall\epsilon>0, \qquad H(E_d/K)^{1-\epsilon} \ll_{q, \epsilon} |\sha(E_d/K)|\cdot \Reg(E_d/K) \ll_{q, \epsilon} H(E_d/K)^{1+\epsilon}.\]
The upper bound essentially proves a conjecture of Lang (originally formulated for elliptic curves over $\Q$ in \cite[Conj. 1]{Lang_conjecturesEC}), and our lower bound reveals that the exponent $1$ is optimal (\ie{} no smaller number would do in the upper bound).
Furthermore, it follows from the computation of $H(E_d/K)$ (see section \ref{sec.invariants}) that one has
 \[ \log\big( |\sha(E_d/K)|\cdot \Reg(E_d/K)\big)  \sim \frac{\log q}{2} \cdot d, 
 \qquad (\text{as }d\to\infty),\]
showing that the product $|\sha(E_d/K)|\cdot \Reg(E_d/K)$ grows exponentially fast with $d$. In the interpretation of \cite{Hindry_MW}, this suggests that the Mordell--Weil groups $E_d(K)$ are ``exponentially hard to compute''.

Note that $\{E_d\}$ is but the second known sequence of elliptic curves satisfying $\lim_{d\to\infty} \BS(E_d/K) =1$ unconditionally (see also \cite[Thm. 1.4]{HP15}). Four more examples were constructed in \cite{Griffon_PHD}.

\medskip

To conclude this introduction, let us give the plan of the paper, as well as a rough sketch of the proof of Theorem \ref{theo.BS.Leg}. The Legendre elliptic curves $E_d$ have been previously studied by Ulmer, Concei{\c{c}}ão and Hall in a series of papers (\cite{Ulmer_legI}, \cite{Ulmer_legII}, ...). 
In particular, they 
proved that $E_d$ satisfies the Birch and Swinnerton-Dyer conjecture (henceforth abbreviated as BSD). This implies the finiteness of $\sha(E_d/K)$, and is the main reason why our Theorem \ref{theo.BS.Leg} is unconditional (see section \ref{sec.invariants}). 
Moreover, 
they have given an explicit expression of the $L$-function $L(E_d/K, T)\in\Z[T]$ of $E_d$ and of its zeroes (see section \ref{sec.Lfunc}).

Our first step towards Theorem~\ref{theo.BS.Leg} will be 
to reduce it to an analytic statement (see \eqref{eq.i.BSvalspe} below). More precisely, denoting by $\rho = \ord_{T=q^{-1}} L(E_d/K, T)$, the BSD conjecture gives the following expression of the special value $L^\ast(E_d/K, 1)$ of the $L$-function of $E_d$: 
\[ L^\ast(E_d/K, 1) := \left.\frac{L(E_d/K, T)}{(1-qT)^{\rho}}\right|_{T=q^{-1}}
=\frac{|\sha(E_d/K)|\cdot \Reg(E_d/K)}{H(E_d/K)}\cdot \left( \text{ extra terms }\right).
\] 
Estimating the size of the ``extra terms'',  we show (see Corollary \ref{coro.rel.spval.BS}) that
\begin{equation}\label{eq.i.BSvalspe} \frac{\log\big(|\sha(E_d/K)\cdot \Reg(E_d/K)|\big)}{\log H(E_d/K)} = 1 +\frac{\log L^\ast(E_d/K, 1)}{\log H(E_d/K)}  +o(1) \qquad (\text{as }d\to\infty).\end{equation}
The size of these ``extra terms'' was first controlled in \cite{HP15} for abelian varieties $A$ over $K$ (see Theorems 1.22 and 3.8 there). However, their proof is quite involved.  
Since the proof in the special case of $E_d$ is elementary and explicit, we thought it was worth giving details here. 
Given \eqref{eq.i.BSvalspe}, proving Theorem \ref{theo.BS.Leg} boils down to showing that 
\begin{equation}\label{eq.i.requiredbound}
-o(1) \leq \frac{\log L^\ast(E_d/K, 1)}{\log H(E_d/K)} \leq o(1) \qquad (\text{as }d\to\infty).\end{equation}

The upper bound in \eqref{eq.i.requiredbound} is relatively easy to prove (see Theorem \ref{theo.spval.UBnd}), but the proof of the lower bound is much more delicate. We proceed as follows:  by definition, the special value $L^\ast(E_d/K, 1)$ has the following shape:
\begin{equation}\label{eq.i.spval}
L^\ast(E_d/K, 1) = \frac{(\text{a positive prime-to-$p$ integer})}{q^{e_q(d)}}, \quad \text{ for some exponent }e_q(d)\geq 0.
\end{equation}
A straightforward estimate shows that $e_q(d) \ll d$, but the lower bound in \eqref{eq.i.requiredbound} requires to prove the stronger statement that $e_q(d)/d \to 0$, as $d\to\infty$. 
This improved bound on $e_q(d)$ constitutes our main technical result (Theorem \ref{theo.spval.LBnd}), the proof of which is given in  section~\ref{sec.bnd.spval}. There are two main ingredients in the proof of this theorem. 
First, we rely on the explicit factorization of $L(E_d/K, T)$ in \cite{Ulmer_legII}  
to obtain an expression of $L^\ast(E_d/K, 1)$ (see Proposition \ref{lemm.expr.spval}). 
Noting that $L^\ast(E_d/K, 1)$ is given in terms of Jacobi sums, we use a variant of Stickelberger's theorem to obtain a reasonably explicit expression of $e_q(d)$.
Second, we observe that the size of the resulting expression of $e_q(d)$  can be estimated by using an average equidistribution result for subgroups of $(\Z/d\Z)^\times$, proved by the author in \cite{Griffon_FermatS}.

\medskip

For the purpose of clarity, we have only stated qualitative bounds in this introduction, but note that we will actually prove a quantitative version of Theorem \ref{theo.BS.Leg}  (see Theorem \ref{theo.BS.Leg.strong}): unlike the (lower bound in the) classical Brauer--Siegel theorem, Theorem \ref{theo.BS.Leg} is entirely effective.

\paragraph{Notations} 

For all integers $d\geq 2$, coprime to $q$, let $\mu_d$ be the group of $d$-th roots of unity in $\bar{\F_q}$.
The cardinality of a finite set $X$ will be denoted by $|X|$.
For any prime power $q$, and any integer $n\geq 1$ coprime to $q$, $\sbgp{q}{n}$ will denote the subgroup of $(\Z/n\Z)^\times$  generated by $q$, and we let $o_q(n):=|\sbgp{q}{n}|$ the multiplicative order of $q$ modulo $n$.
For two functions $f(x), g(x)$ defined on $[0,\infty)$, we make use of the Vinogradov notation $f(x)\ll_{a} g(x)$ to mean that there exists a constant $C>0$ (depending at most on the mentioned parameters $a$) such that $|f(x)|\leq C g(x)$ for $x\to\infty$. 
Unless otherwise stated, all constants are effective and could be made explicit.

\paragraph{Acknowledgements}
This paper contains results from the author's PhD thesis \cite{Griffon_PHD}. 
The author wishes to thank his advisor Marc Hindry for many illuminating conversations and comments. 
He also thanks Michael Tsfasman and Douglas Ulmer  for their careful reading of an earlier version of this text, and their suggestions. 
He would also like to thank 
the Universiteit Leiden, for providing ideal working conditions during the writing of this article.

\numberwithin{equation}{section}

\section{The Legendre elliptic curves}\label{sec.invariants}

\begin{center}
\emph{Throughout the paper, we fix a finite field $\F_q$ of odd characteristic $p\geq 3$
and we denote by $K=\F_q(t)$.}
\end{center}

For any integer $d\geq 2$, coprime to $p$, we consider the  \emph{Legendre elliptic curve} $E_d/K$, given by the affine Weierstrass model
\begin{equation}\label{eq.Wmod}
E_d: \qquad  y^2=x(x+1)(x+t^d).
\end{equation}
The discriminant of this model of $E_d$ is easily seen to be $\Delta = 16 t^{2d}(t^d-1)^2$.
Likewise, the $j$-invariant of $E_d$ is easily computed from \eqref{eq.Wmod} and we find:  \[j(E_d/K) = \frac{2^8\cdot (t^{2d}-16t^d+1)^3}{t^{2d}(t^d-1)^2}\in K.\]
We note that $j(E_d/K)$ is nonconstant, so that $E_d$ is not isotrivial. Furthermore, $j(E_d/K)\notin K^p$.

\begin{rema} We follow \cite{Ulmer_legI} in calling $E_d$ a Legendre curve:
see \cite[\S2]{Ulmer_legI} for more comments on this choice of terminology.
We also note the slight change in points of view compared to \cite{Ulmer_legI}, \cite{Ulmer_legII}: instead of considering a fixed curve $E_1$ over a varying field $\F_q(t^{1/d})$, we fix the base field $\F_q(t)$ and vary the curve $E_d$. This is only a matter of convenience, and has no influence on the results.
\end{rema}

This section is mainly expository and does not contain new results: we review the definitions  of the invariants of $E_d$ and the computations of some of them, we also state the relevant theorems about $E_d$.
In the last subsection, we explain how the problem of bounding $|\sha(E_d/K)|\cdot\Reg(E_d/K)$ can be reduced to studying the size of the special value  $L^\ast(E_d/K, 1)$ of the $L$-function of $E_d/K$ at $s=1$.

\subsection{Review of the invariants}

Applying Tate's algorithm 
(as in \cite[Chap. IV, \S9]{ATAEC} for example), one proves that the bad reduction of $E_d$ is as follows:

\begin{prop}\label{prop.badred}
Let $|\mu_d|$ be the set of places of $K$ that divide $t^d-1$ (\ie{} $|\mu_d|$ is the set of closed points of $\P^1$ corresponding to $d$-th roots of unity).
Then $E_d/K$ has good reduction outside the places of $K$ corresponding to $S=\{0\}\cup\mu_d\cup\{\infty\}\subset \P^1$; moreover,
the bad reduction of $E_d$ is as follows:
\begin{center}\renewcommand{\arraystretch}{2.0}
\begin{tabular}{| c | c   l | c | c | c |} 
\hline
Place $v$ of $K$&  \multicolumn{2}{c|}{Type of $E_d$ at $v$ }  & $\ord_{v}\Delta_{\min}(E_d/K)$ & $\ord_{v}\cond(E_d/K)$ & $ c_v(E_d/K)$ \\
\hline \hline
$0$  & $\type{I}_{2d}$   && $2d$ & $1$ & $2d$ \\ \hline
$v\in|\mu_d|$
 &  $\type{I}_2$    & & $2$ &$1$& $1$ or $2$ \\  \hline
\multirow{ 2}{*}{$\infty$}& $\type{I}_{2d}$   & if $d$ even & $2d$ & $1$ &  $2d$ \\ \cline{2-6}
    & $\type{I}_{2d}^\ast$ &  if $d$ odd & $2d+6$ & $2$ & $4$ \\ \hline
\end{tabular}
\end{center}
\end{prop}

\noindent
In this table, for all places  $v$ of $K$ where $E_d$ has bad reduction, we have denoted by $\ord_{v}(\Delta_{\min})$ (resp. $\ord_v(\cond)$) the valuation at $v$ of the minimal discriminant of $E_d$ (resp. of the conductor of $E_d$), and  by $ c_v(E_d/K)$ the local Tamagawa number (see \cite[Chap. IV, \S9]{ATAEC}). For $v=\zeta$ dividing $t^d-1$ (\ie{} $v$ corresponds to a closed point of $\mu_d$), note that $c_{\zeta}(E_d/K)=2$ if and only if $-1$ is a square in $\F_q(\zeta)$, and $c_{\zeta}(E_d/K)=1$ else.

From this Proposition, it follows easily that the \emph{minimal discriminant divisor} $\Delta_{\min}(E_d/K)$, and the \emph{conductor} $\cond(E_d/K)$ of $E_d$ have degrees:
\begin{equation}\label{eq.inv1}
\deg \Delta_{\min}(E_d/K) = \begin{cases}
6d &\text{ if $d$ is even}, \\
6(d+1) &\text{ if $d$ is odd}, \end{cases}\quad \text{ and } \quad 
\deg \cond(E_d/K)= \begin{cases}
d+2 &\text{ if $d$ is even}, \\
d+3 &\text{ if $d$ is odd}. \end{cases} 
\end{equation}
By definition, 
the \emph{exponential differential height of $E_d/K$} is
\begin{equation}\label{eq.inv2}
H(E_d/K) = q^{\frac{1}{12}\deg\Delta_{\min}(E_d/K)}=q^{\partint{\frac{d+1}{2}}}.
\end{equation}
See \cite[Lemma 7.1]{Ulmer_legI} for a more geometric computation of $H(E_d/K)$ when $d$ is even. Finally, the 
\emph{Tamagawa number} $\tam(E_d/K) := \prod_{v\in S} c_v(E_d/K)$ could be computed exactly from the last column of the table in Proposition \ref{prop.badred},
but we will content ourselves with the estimate
\begin{equation}\label{eq.tam.est}
1 \leq \tam(E_d/K) \leq (2d)^2 \cdot 2^{\theta_q(d)},
\end{equation}
where $\theta_q(d)$ denotes the number of closed points of $\mu_d$, \ie{} $\theta_q(d)$ is the number of monic irreducible polynomials $P\in\F_q[t]$ such that $P\mid t^d-1$.

\subsection{Néron-Tate regulator and Tate--Shafarevich group}
By the Mordell--Weil theorem (proved by Lang and Néron in this setting), the group $E_d(K)$ is finitely generated. 
Moreover, the Mordell--Weil group $E_d(K)$ is endowed with the canonical Néron--Tate height $\hhat{h}_{NT}:E_d(K) \to \Q$. 
Note that, over $\F_q(t)$, it is indeed possible to normalize $\hhat{h}_{NT}$ to have rational values, because it has an interpretation in terms of intersection theory on the minimal regular model of $E_d/K$ (see \cite[Chap. III, \S9]{ATAEC} for details, more specifically Theorem 9.3 there).
The quadratic map $\hhat{h}_{NT}$ induces a $\Z$-bilinear pairing $\langle -, - \rangle_{NT}: E_d(K)\times E_d(K)\to\Q$, which is nondegenerate modulo $E_d(K)\tors$ (\cf{} \cite[Chap. III, Thm. 4.3]{ATAEC}). 
The \emph{Néron-Tate regulator} of $E_d/K$ is then defined to be the Gram determinant
\[\Reg(E_d/K) := \left|\det\left( \langle P_i, P_j \rangle\right)_{1\leq i,j \leq r}\right| \in\Q^\ast,\]
for any choice of a $\Z$-basis $P_1, \dots, P_r \in E_d(K)$ of $E_d(K)/E_d(K)\tors$. 
Let us also recall that the \emph{Tate--Shafarevich group} of $E_d/K$
is defined by
\[\sha(E_d/K) :=
\ker\left( \H^1(K, E_d) \longrightarrow \prod_{v} \H^1(K_v, (E_d)_v)\right),\]
where the involved cohomology groups are Galois cohomology groups 
(see \cite[Chap. X, \S4]{AEC} for more details). 
We will see in Theorem \ref{theo.BSD.Leg} that $\sha(E_d/K)$ is a finite group, which will prove the first assertion of Theorem \ref{theo.BS.Leg}.

\subsection{BSD conjecture and consequences}

The Hasse--Weil $L$-function of $E_d/K$ is \emph{a priori} defined as a formal Euler product over the places $v$ of $K$:
\[L(E_d/K, T) = \prod_{v\text{ good}} (1-a_v(E_d)\cdot T^{\deg v} + q^{\deg v}\cdot T^{2\deg v})^{-1} \cdot \prod_{v\text{ bad}} (1-a_v(E_d)\cdot T^{\deg v})^{-1},\]
where $a_v(E_d)$ is defined by counting rational points on the reduction of $E_d$ modulo $v$ 
(see \cite[Appendix C, \S16]{AEC} or \cite[Appendix]{Brumer} 
for more details). 

A deep theorem of Grothendieck shows 
that $L(E_d/K, T)$ is actually a polynomial in $T$, with integral coefficients, and whose degree is given explicitely in terms of $\cond(E_d/K)$. Moreover, by Deligne's proof of the Riemann Hypothesis, the zeroes of $L(E_d/K, T)$ have magnitude $q^{-1}$ in any complex embedding. 
We can thus study the behaviour of $L(E_d/K, T)$ around the point $T=q^{-1}$, and introduce the \emph{special value of $L(E_d/K, T)$ at $T=q^{-1}$}:
\begin{equation}\label{eq.def.spval}
L^\ast(E_d/K, 1) := \left. \frac{L(E_d/K, T)}{(1-qT)^\rho}\right|_{T=q^{-1}} \in \Z[q^{-1}]\smallsetminus\{0\} 
\quad \text{ where } \rho = \ord_{T=q^{-1}} L(E_d/K, T).
 \end{equation}

 Inspired by the conjecture of Birch and Swinnerton-Dyer for elliptic curves over $\Q$, Tate \cite{Tate_BSD} conjectured that $\rho$ and $L^\ast(E_d/K, 1)$ have an arithmetic interpretation. Their conjecture has been proved for $E_d$ by Ulmer (\cf{} \cite[Coro. 11.3]{Ulmer_legI}), and we state the result as follows:

\begin{theo}[Ulmer]\label{theo.BSD.Leg}
 Let $\F_q$ be a finite field of odd characteristic and $K=\F_q(t)$. For all integers $d\geq 1$, coprime with $q$, let $E_d/K$ be the Legendre curve \eqref{eq.Wmod} as above. Then the BSD conjecture is true for $E_d/K$; that is to say, 
\begin{enumerate}[(1)]
\item the Tate--Shafarevich group $\sha(E_d/K)$ is finite,
\item the rank of $E_d(K)$ is equal to $\rho = \ord_{T=q^{-1}} L(E/K, T)$,
\item and one has
\begin{equation}\label{eq.BSD}
L^\ast(E_d/K, 1) = \frac{|\sha(E_d/K)|\cdot \Reg(E_d/K)}{H(E_d/K)}\cdot \frac{\tam(E_d/K)\cdot q}{|E_d(K)\tors|^2}.\end{equation}
\end{enumerate}
\end{theo}

The proof goes roughly as follows (see  \cite[\S11]{Ulmer_legI} and \cite[\S7]{Ulmer_MWJacobians} for more details). 
We denote by $\pi:\Ecal_d\to\P^1$ the minimal regular model of $E_d$. 
By the main theorem of \cite{KatoTrihan}, it suffices to prove \emph{(2)} (the ``weak BSD conjecture'').  
In turn, proving the equality in \emph{(2)} is equivalent, by \cite{Tate_BSD} and \cite{Milne_conjAT}, to proving that the Tate conjecture holds for the surface $\Ecal_d$.
Ulmer \cite[\S7]{Ulmer_legI} has explicitely constructed the model 
$\Ecal_d\to\P^1$ and by \cite{Berger}, the corresponding surface $\Ecal_d$ is dominated by a product of curves. 
The Tate conjecture has been proved for products of curves (see \cite{Tate_conj}), and the existence of a dominant map to $\Ecal_d$ implies the truth of this conjecture for $\Ecal_d$. 

 \medskip

The link between the product $|\sha(E_d/K)|\cdot \Reg(E_d/K)$ and the special value
$L^\ast(E_d/K, 1)$ is now clear. 
For any integer $d\geq 2$ coprime with $q$, reordering terms in \eqref{eq.BSD} and taking a $\log$, we obtain that:
\begin{equation}\label{eq.rel.spval}
\frac{\log \big(|\sha(E_d/K)|\cdot \Reg(E_d/K)\big)}{\log H(E_d/K)} = 1+ \frac{\log L^\ast(E_d/K, 1)}{\log H(E_d/K)} + \frac{\log |E_d(K)\tors|^2 - \log(\tam(E_d/K)\cdot q)}{\log H(E_d/K)}.\end{equation}
Let us show that the right-most term is asymptotically negligible. 
First of all, the explicit computation of $E_d(K)\tors$ in \cite[Proposition 6.1]{Ulmer_legI} implies that $|E_d(K)\tors| \leq 8$. 
Secondly, \eqref{eq.inv2} and \eqref{eq.tam.est} lead to 
\[ \frac{\log(\tam(E_d/K)\cdot q) }{\log H(E_d/K)} \ll_q \frac{\log d}{d} + \frac{\theta_q(d)}{d} + \frac{1}{d}.  \]
Remember that $\theta_q(d)$ is the number of monic irreducible polynomials in $\F_q[t]$ which divide $t^d-1$. 
Since $t^d-1$ factors as a product of cyclotomic polynomials $\Phi_{d'}(t) \in\F_q[t]$ (with $d'\mid d$), and since $\Phi_{d'}(t)$ decomposes as a product of $\phi(d')/o_q(d')$ distinct monic irreducible factors in $\F_q[t]$ (see \cite[Chap. 13, \S2]{IR}), we have
$\theta_q(d) = \sum_{d'\mid d} \frac{\phi(d')}{o_q(d')}$. 
In the course of the proof of Lemma \ref{lemm.prel}\eqref{item.ii.actq}  below, we will see that 
 $\sum_{d'\mid d} \frac{\phi(d')}{o_q(d')} \ll_q d/\log d$.
Consequently, ${\log(\tam(E_d/K)\cdot q) }/{\log H(E_d/K)}\ll_q (\log d)^{-1}$. 

Transfering this last estimate and that on $|E_d(K)\tors|$   into \eqref{eq.rel.spval} immediately yields the following:

\begin{coro} \label{coro.rel.spval.BS}
For all integers $d\geq 2$, coprime with $q$, we have:
\[ \frac{\log\big(|\sha(E_d/K)|\cdot\Reg(E_d/K)\big)}{\log H(E_d/K)} = 1 +\frac{\log L^\ast(E_d/K, 1)}{\log H(E_d/K)} + O_q\left(\frac{1}{\log d}\right) \quad (\text{as }d\to\infty),\] 
where the implicit constant is effective and depends at most on $q$.
\end{coro}

\begin{rema} 
This corollary is but an explicit version of a special case of a result in 
\cite{HP15}. 
 In particular, see the discussion in \cite[\S2]{HP15} where Corollary \ref{coro.rel.spval.BS} is proved for abelian varieties over $K$ satisfying BSD. Note that the proof in the general case is much more involved: it requires delicate diophantine estimates on the torsion subgroup (\cite[Theorem 3.8]{HP15}) and on Tamagawa numbers (\cite[Theorem 6.5]{HP15}).
\end{rema}

\section[Expressions of the L-function and of the special value]{The $L$-function of $E_d$ and its special value}
\label{sec.Lfunc}

We have reduced the estimation of $|\sha(E_d/K)\cdot \Reg(E_d/K)|$ to that of $L^\ast(E_d/K, 1)$. As we explained in the introduction, we need to make use of the specific structure of the $L$-function of $E_d$ to obtain the correct estimate of $L^\ast(E_d/K, 1)$.
In this section, we introduce the notations required to state the explicit expression for $L(E_d/K, T)$ (obtained in \cite{Ulmer_legII}) and we then proceed to express $L^\ast(E_d/K, 1)$ under a suitable form. The proof of the lower bound in itself will be the goal of the next section.
\subsection[Action]{Action of $q$ on $\Z/d\Z$}
Let $\F_q$ be a finite field of odd characteristic, and $d\geq 2$ be an integer coprime to $q$. There is a natural action of $q$ on $\Z/d\Z$ by $n\mapsto q\cdot n$. 
For any subset $Z\subset \Z/d\Z$ which is stable by multiplication by $q$, we denote by $\O_q(Z)$ the set of orbits of $Z$ under this action. 
In what follows, we will be particularly interested in the set
\[Z_d := \begin{cases}\Z/d\Z \smallsetminus\{0, d/2\} &\text{ if $d$ is even,} \\
 \Z/d\Z \smallsetminus\{0\} &\text{ if $d$ is odd,}
 \end{cases}
\]
(which is stable under multiplication by $q$) and in the corresponding set of orbits $\O_q(Z_d)$. 
Given an orbit $\Mm\in\O_q(Z_d)$, we will often have to make a choice of representative $m\in Z_d$ of this orbit: we thus stick to the useful convention that orbits in $\O_q(Z_d)$ are always denoted by a bold letter ($\Mm$, $\mathbf{n}$,~...) and that the corresponding normal letter ($m$, $n$,~...) designates any choice of representative of this orbit in $Z_d$.

For any orbit $\Mm\in\O_q(Z_d)$, its length $|\Mm| = \left| \left\{ m, qm, q^2m, \dots\right\}\right|$ is clearly equal to
\[|\Mm| = \min\left\{ \nu\in\Z_{\geq 1} \ \big| \ q^\nu m \equiv m \bmod{d}\right\},\]
which, in turn, equals 
$|\Mm| = o_q(d/\gcd(d, m))$, the multiplicative order of $q$ modulo $d/\gcd(d,m)$, for any $m\in\Mm$.
For any power $q^v$ of $v$, by construction of the multiplicative order, remark that $q^v m\equiv m \bmod{d}$ if and only if $|\Mm|$ divides $v$, \ie{} if and only if $\F_{q^v}$ is a finite extension of $\F_{q^{|\Mm|}}$.

For further use, we record here a few useful facts about the action of $q$ on $Z_d$ in the following lemma:

 \begin{lemm}\label{lemm.prel}
Let $d\geq 2$ be an integer coprime with $q$. The following upper bounds hold:
\begin{multicols}{2}
\begin{enumerate}[(a)]
\item
\label{item.0} 
$\displaystyle \sum_{\substack{e\mid d \\ e\geq 2}}\frac{\phi(e)}{\log e} \ll  \frac{d}{\log d}$,
\item \label{item.i.actq} 
$ \displaystyle \sum_{\Mm\in\O_q(Z_d)}  |\Mm| = |Z_d| \leq d$,
\item \label{item.ii.actq} 
$ \displaystyle \sum_{\Mm\in\O_q(Z_d)}  1 =  |\O_q(Z_d)| \ll \log q\cdot \frac{d}{\log d}$,
\item \label{item.iii.actq} 
$ \displaystyle \sum_{\Mm\in\O_q(Z_d)}  \log |\Mm|  \ll \log q\cdot \frac{d\cdot \log\log d}{\log d}$.
\end{enumerate}
\end{multicols}
All the involved constants are absolute and effective.
\end{lemm}

\begin{proof} 
We put $x=\sqrt{d}$, cut the sum in \eqref{item.0} into two parts, and bound them separately: this leads to
\begin{align*}
\sum_{\substack{e\mid d \\ e\geq 2}}\frac{\phi(e)}{\log e}
& = \sum_{\substack{e\mid d \\ 2\leq e\leq x}}\frac{\phi(e)}{\log e} + \sum_{\substack{e\mid d \\ x <e }}\frac{\phi(e)}{\log e}
\leq \sum_{\substack{e\mid d \\ 2\leq e\leq x}}\frac{e}{\log e} 
+  \frac{1}{\log x}\sum_{e\mid d}\phi(e)\\
&\leq \frac{x}{\log x}\sum_{2\leq e\leq x}1
+  \frac{1}{\log x}\sum_{e\mid d}\phi(e)
\leq \frac{x^2+d}{\log x} \leq 4\cdot \frac{d}{\log d},
\end{align*}
because $y\mapsto 1/\log y$ is decreasing on $[3, \infty[$, while $y\mapsto y/\log y$ is increasing on $[3, +\infty[$. This proves \eqref{item.0} with an implicit constant $c_0\leq 4$.
Assertion \eqref{item.i.actq} is a direct consequence of the fact that $Z_d$ can be written as the disjoint union of its orbits under the action of $q$ by multiplication. 

For any divisor $d'$ of $d$, let $Y_{d'} := \left\{ n\in Z_d \mid \gcd(n,d)=d/d'\right\}$ (note that  $Y_{1}=\varnothing$, and that $Y_{2}=\varnothing$ if $d$ is even).
Since $\gcd(d,q)=1$, these sets $Y_{d'}\subset Z_d$ are stable under the action of $q$, and  the decomposition $Z_d=\bigsqcup_{d'\mid d}Y_{d'}$ of $Z_d$ induces a corresponding decomposition 
$\O_q(Z_d) = \bigsqcup_{d'\mid d}\O_q(Y_{d'})$. 
For all $d'\mid d$ with $d'>2$, elements $n\in Y_{d'}$ are of the form $n=t\cdot d/d' $, where $t\in(\Z/d'\Z)^\times$, so that the map $t\mapsto t\cdot d/d'$ induces a bijection 
between the quotient group $(\Z/d'\Z)^\times/\sbgp{q}{d'}$ and $\O_q(Y_{d'})$.
This implies that $|\O_q(Y_{d'})|=\phi(d')/o_q(d')$ (where $o_q(d')=|\sbgp{q}{d'}|$ denotes the multiplicative order of $q$ modulo $d'$), but also that all the orbits $\Mm\in \O_q(Y_{d'})$ have the same cardinality $|\Mm| = o_q(d')$.
By definition of the order, $d'$ divides $q^{o_q(d')}-1$, hence $o_q(d')\geq \log d'/\log q$. We thus deduce from \eqref{item.0} that
\[|\O_q(Z_d)| 
= \sum_{\substack{d'\mid d \\ d'>2}} |\O_q(Y_{d'})| 
= \sum_{\substack{d'\mid d \\ d'>2}} \frac{\phi(d')}{o_q(d')} 
\leq \log q \cdot \sum_{\substack{d'\mid d \\ d'>2}} \frac{\phi(d')}{\log d'}  
\leq c_0 \log q \cdot \frac{d}{\log d}.
\]
It remains to prove the last assertion \eqref{item.iii.actq}: using the remarks in the previous paragraph, we can write
\[\sum_{\Mm\in\O_q(Z_d)} \log|\Mm| 
= \sum_{\substack{d'\mid d\\ d'>2}} \sum_{\Mm\in\O_q(Y_d')}\log |\Mm| 
= \sum_{\substack{d'\mid d\\ d'>2}} \frac{\phi(d')}{o_q(d')}\cdot\log o_q(d').\]
Given a parameter $\theta\in ]1,o_q(d)[$, we split this last sum according to the size of $o_q(d')$ compared to $\theta$: we then bound from above the two resulting sums, using that $y\mapsto \log y$ is increasing, that $y\mapsto (\log y)/y$ is decreasing, and adding nonnegative terms:
\[
\sum_{\substack{d'\mid d\\ d'>2}} \frac{\phi(d')}{o_q(d')}\cdot\log o_q(d')
\leq \log \theta \cdot \sum_{\substack{d'  \text{ s.t.} \\ o_q(d')\leq \theta}} \frac{\phi(d')}{o_q(d')}
+ \frac{\log \theta}{\theta}\cdot \sum_{\substack{d' \text{ s.t.} \\  o_q(d')> \theta}} \phi(d')
\leq \log \theta \cdot |\O_q(Z_d)|
+ \frac{\log \theta}{\theta}\cdot d.
\]
Upon choosing 
$\theta = \log d/(c_0 \log q)$ (note that $ \theta \leq o_q(d)/c_0 < o_q(d)$), we deduce from \eqref{item.ii.actq} that
\[\sum_{\Mm\in\O_q(Z_d)} \log|\Mm| 
\leq \log \theta \cdot d\cdot \left( \frac{  c_0 \log q}{\log d}  + \frac{1}{\theta}  \right)
\leq 2c_0\log q \cdot \frac{d\cdot\log\log d}{\log d}
,\]
which is inequality \eqref{item.iii.actq}. 
\ProofEnd\end{proof}

\subsection{Jacobi sums} \label{sec.jacobi}
We fix, once and for all, an algebraic closure $\bar{\Q}$ of $\Q$ (of which all number fields are assumed to be subfields) and a prime ideal $\gP$ above $p$ in the ring integers $\bar{\Z}$ of $\bar{\Q}$.
The residue field $\bar{\Z}/\gP$ is an algebraic closure $\bar{\F_p}$ of $\F_p$ (and all finite fields of characteristic $p$ are seen as subfields thereof).
The reduction map $\bar{\Z}\to\bar{\Z}/\gP$ induces an isomorphism between the group $\mu_{\infty, p'}\subset\bar{\Z}^\times$ of roots of unity of order prime to $p$ and the multiplicative group $\bar{\F_p}^\times$. 
We let $\tt:\bar{\F_p}^\times\to\mu_{\infty, p'}$ 
be the inverse of this isomorphism, and we denote by the same letter the restriction of $\tt$ to any finite field $\F_q$. 

For any finite extension $\F_Q$ of $\F_p$, we denote by $\lambda_Q:\F_Q^\times\to\{\pm1\}$ 
the unique nontrivial character of order $2$ of $\F_Q^\times$ (the ``Legendre symbol'' of $\F_Q$), extended by $\lambda_Q(0):=0$.
For any $m \in Z_d$, we define a character $\tt_m : \F_{q^{|\Mm|}}^\times\to \Qbar^\times$ by
\[\forall x\in\F_{q^{|\Mm|}}^\times, \quad  \tt_m(x) = \tt(x)^{(q^{|\Mm|}-1)m/d} \quad \text{ and we let } \tt_m(0):=0. 
\]
By construction, the characters $\tt_m$ are nontrivial, have order dividing $d$ (in fact, the order of $\tt_m$ is exactly $d/\gcd(d,m)$) and, as such, take values in $\Q(\zeta_d)$, the $d$-th cyclotomic field. 
 
\medskip

To  $m \in Z_d$, we can now attach  a Jacobi sum:
\begin{equation}\label{eq.def.jacobi}
\Ja(m) 
= \sum_{x\in\F_{q^{|\Mm|}}} \tt_m(x)\cdot\lambda_{q^{|\Mm|}}(1-x).
\end{equation}
As a sum of $d$-th roots of unity, $\Ja(m)$ is an algebraic integer in $\Q(\zeta_d)$. Even though $\tt_m$ might differ from $\tt_{q\cdot m}$, it turns out that $\Ja(m) = \Ja(q\cdot m)$ because $x\mapsto x^q$ is a bijection of $\F_{q^{|\Mm|}}$ (more generally, $\Ja(m) = \Ja(p\cdot m)$). Thus it makes sense to associate a Jacobi sum $\Ja(\Mm)$ to any orbit $\Mm\in\O_q(Z_d)$:  we let $\Ja(\Mm)=\Ja(m)$ for any choice of $m\in\Mm$. 
Since, for all $m\in Z_d$, none of $\tt_m, \lambda_{q^{|\Mm|}}$ and $\tt_m\cdot \lambda_{q^{|\Mm|}}$ is trivial, it is well-known that $|\Ja(\Mm)|=q^{|\Mm|/2}$. 
The reader may consult \cite{IR} for more details about Jacobi sums.

\subsection[L-function and special value]{$L$-function and special value}

As above, for any integer $d$, we let 
\[Z_d := \begin{cases}\Z/d\Z \smallsetminus\{0, d/2\} &\text{ if $d$ is even,} \\
 \Z/d\Z \smallsetminus\{0\} &\text{ if $d$ is odd },
 \end{cases}
\]
and $\O_q(Z_d)$ be the set of orbits of $Z_d$ under the action of $q$ by multiplication. With the notations introduced in the last two subsections (which are essentially the same as those of \cite[\S3]{Ulmer_legII}), we can now state \cite[Theorem 3.2.1]{Ulmer_legII}:

\begin{theo}[Concei{\c{c}}ão, Hall, Ulmer]\label{theo.Lfunc}
Let $d\geq 2$ be an integer coprime with $q$. The $L$-function of $E_d/K$ is given by
\begin{equation}\label{eq.Lfunc}
L(E_d/K, T)  = \prod_{\Mm \in\O_q(Z_d)} \left(1-\Ja(\Mm)^2 \cdot T^{|\Mm|}\right),
\end{equation}
where $\Ja(\Mm)$ is the Jacobi sum defined in \eqref{eq.def.jacobi}.
\end{theo}

The proof of \eqref{eq.Lfunc} in \cite[\S3]{Ulmer_legII} hinges on a clever manipulation of character sums. 
Since the minimal regular model of $E_d/K$ is explicitely known (see \cite[\S7]{Ulmer_legI}), the computation can also be done via cohomological methods. Though less elementary, the latter has the advantage of ``explaining'' the appearance of Jacobi sums in the $L$-function. 

\medskip

Given \eqref{eq.Lfunc}, it is now easy to give an explicit expression for the special value $L^\ast(E_d/K, 1)$ of the $L$-function of $E_d/K$ at $s=1$. We begin by introducing the following two subsets $V_d$ and $S_d$ of $Z_d$:
\[ V_d := \left\{ m \in Z_d \ \big| \ \Ja(m)^2=q^{|\Mm|}\right\} \quad \text{ and } \quad
S_d :=Z_d\smallsetminus V_d. \]
By their very construction, the sets $V_d$ and $S_d$ are stable under the action of $q$ on $Z_d$. As we will see in the Proposition below, the orbit set $\O_q(V_d)$ (resp. $\O_q(S_d)$) parametrizes the factors in \eqref{eq.Lfunc} that vanish at $T=q^{-1}$ (resp. those that give a nontrivial contribution to the special value).

 With this notation at hand, we prove:
\begin{prop}\label{lemm.expr.spval}
 For any integer $d\geq 2$ prime to $q$, the special value $L^\ast(E_d/K, 1)$ admits the following expression:
\begin{equation}\label{eq.expr.spval}
L^\ast(E_d/K, 1) = \prod_{\Mm \in \O_q(V_d)} |\Mm| \cdot \prod_{\Mm\in\O_q(S_d)} \left(1-\frac{\Ja(\Mm)^2}{q^{|\Mm|}}\right).
\end{equation}
\end{prop}
\begin{proof} 
For any $\Mm\in\O_q(Z_d)$, let $g_\Mm(T) := 1- \Ja(\Mm)^2\cdot T^{|\Mm|}$ be the corresponding factor of $L(E_d/K, T)$ (see Theorem \ref{theo.Lfunc}). 
A straightforward computation shows that we have
\[\rho' = \ord_{T=q^{-1}}g_\Mm(T) =\begin{cases}
1 &\text{if } \Mm\in\O_q(V_d), \\
0 &\text{otherwise, }  
\end{cases}
\quad \text{and} \quad
\left. \frac{g_{\Mm}(T)}{(1-qT)^{\rho'}}\right|_{T=q^{-1}} 
=\begin{cases}
|\Mm| &\text{if } \Mm\in\O_q(V_d), \\
1 - \frac{\Ja(\Mm)^2}{q^{|\Mm|}} &\text{otherwise. }  
\end{cases} 
\]
By definition of $L^\ast(E_d/K, 1)$ (see \eqref{eq.def.spval}), the desired expression follows by taking the product over all $\Mm \in\O_q(Z_d) = \O_q(V_d)\sqcup\O_q(S_d)$ of the ``special values'' at $T=q^{-1}$ of the polynomials $g_\Mm(T)$.
\ProofEnd\end{proof}

 \begin{rema} \label{rema.rank}
    By Theorem \ref{theo.BSD.Leg}, we know that the rank of $E_d(K)$ is equal to $\rho_d := \ord_{T=q^{-1}}L(E_d/K, T)$.
The proof of the above Proposition implies that $\rho_d =\#\O_q(V_d)\leq \#\O_q(Z_d)$. 
Hence, by Lemma \ref{lemm.prel}\eqref{item.ii.actq}, we have $\rho_d \ll_{q} d/\log d$ (thus recovering 
Brumer's bound on the analytic rank
\cite[Prop. 6.9]{Brumer}). 
 \end{rema}

\subsection{Upper bound on the special value}
Let us prove an upper bound on $L^\ast(E_d/K,1)$:
\begin{theo} \label{theo.spval.UBnd}
Let $\F_q$ be a finite field of odd characteristic and $K=\F_q(t)$. For all integer $d\geq 2$ prime to $q$, the special value $L^\ast(E_d/K, 1)$ satisfies the upper bound:
\begin{equation}\label{eq.spval.UBnd}
\frac{\log L^\ast(E_d/K, 1)}{\log H(E_d/K)} \leq A\cdot \frac{\log \log d}{\log d}, 
\end{equation}
for some effective absolute constant $A>0$.
\end{theo}

\begin{rema}
The bound \eqref{eq.spval.UBnd} is not better than the ``generic'' upper bound of  \cite[Thm. 7.5]{HP15} on special values of $L$-functions of abelian varieties over $K$. 
The bound in \cite[Thm. 7.5]{HP15} is proved with methods from classical complex analysis. 
We include a proof of \eqref{eq.spval.UBnd} 
nonetheless, 
because our proof is more elementary 
and gives 
a very explicit estimate. 
\end{rema}

\begin{proof} From \eqref{eq.expr.spval} and the fact that $|\Ja(\Mm)|^2=q^{|\Mm|}$ for all $\Mm\in\O_q(Z_d)$, the triangle inequality leads to
\[  \log L^\ast(E_d/K, 1)  
=  \sum_{\Mm\in\O_q(V_d)} \log |\Mm| + \sum_{\Mm\in\O_q(S_d)}\log\left|1 - \frac{\Ja(\Mm)^2}{q^{|\Mm|}}\right| 
\leq  \sum_{\Mm\in\O_q(Z_d)} \log |\Mm| + \log 2 \cdot |\O_q(Z_d)|.\]
Both the sum on the right-hand side and $|\O_q(Z_d)|$ have already been bounded from above in Lemma \ref{lemm.prel} (items \eqref{item.ii.actq} and \eqref{item.iii.actq}). We thus infer that 
\[ \frac{\log L^\ast(E_d/K, 1)}{d\cdot\log q}
\ll \frac{\log\log d}{\log d} + \frac{1}{\log d} 
\ll \frac{\log\log d}{\log d}.\]
And since, by \eqref{eq.inv2}, one has $\log H(E_d/K) = \partint{\frac{d+1}{2}}\cdot \log q$, we conclude that
\[\frac{\log L^\ast(E_d/K, 1)}{\log H(E_d/K)} 
\leq  \frac{\log L^\ast(E_d/K, 1)}{d\cdot \log q}\cdot \frac{2d}{ d-1}
\ll  \frac{\log \log d}{\log d}.\]
\ProofEnd\end{proof}

\section{Lower bound on the special value}\label{sec.bnd.spval}

Let $d\geq 2$ be a integer, coprime to $q$. By construction (see \eqref{eq.def.spval}), the special value $L^\ast(E_d/K, 1)$ is the value of a certain polynomial $L^\ast_d(T)\in \Z[T]$ at $T=q^{-1}$. Since $L^\ast_d(T)$ does not vanish at $T=q^{-1}$, one has $|L^\ast(E_d/K, 1)|\geq q^{-\deg L^\ast_d(T)}$. Furthermore, by \eqref{eq.inv1} and by Remark \ref{rema.rank}  (or by Brumer's bound on the analytic rank \cite[Prop. 6.9]{Brumer}), one has
\[\deg L^\ast_d(T) = \deg L(E_d/K, T) - \ord_{T=q^{-1}} L(E_d/K, T) =  d+ o(d) \qquad (\text{as } d\to\infty).\]
This quick argument 
already yields the following lower bound on $L^\ast(E_d/K, 1)$: 
\begin{equation}\label{eq.triv.LBnd}  \frac{\log |L^\ast(E_d/K, 1)|}{d\cdot \log q}   \geq -1 +o(1)\qquad (\text{as } d\to\infty). \end{equation}

However, computational evidence suggests that this ``trivial'' lower bound on $L^\ast(E_d/K, 1)$ is far from the truth. 
In some special instances, one can improve on \eqref{eq.triv.LBnd}. For example, when $d$ is of the form $d = q^n+1$ (with $n\to\infty$), a theorem of Shafarevich and Tate  shows that  $\Ja(m)^2=q^{|\Mm|}$ for all $m\in Z_d$ (see \cite{ShaTate_Rk}, \cite[Prop. 8.1]{Ulmer_LargeRk}). In the notations of Proposition \ref{lemm.expr.spval}, this means that $V_d= Z_d$ and $S_d=\varnothing$. 
Thus, for these $d$'s, the special value $L^\ast(E_d/K, 1)$ is actually a positive integer and we obtain an 
improved lower bound:
\begin{equation}\label{eq.supersing.LBnd}  
\frac{\log |L^\ast(E_d/K, 1)|}{d\cdot \log q}   \geq 0\qquad (\text{when } d=q^n+1, \text{with }n\to\infty). \end{equation}

 In this section, we prove that the stronger \eqref{eq.supersing.LBnd} holds, up to an error term,
 for any $d\geq 2$ coprime with~$q$. 
 More precisely, we will show:

\begin{theo} \label{theo.spval.LBnd}
Let $\F_q$ be a finite field of odd characteristic $p$ and $K=\F_q(t)$. For all integer $d\geq 2$ prime to $q$, the special value $L^\ast(E_d/K, 1)$ satisfies the lower bound: 
\begin{equation}\label{eq.spval.LBnd}
\forall \epsilon\in(0,1/4), \qquad \frac{\log L^\ast(E_d/K, 1)}{\log H(E_d/K)} \geq    -  B\cdot \left(\frac{\log \log d}{\log d}\right)^{1/4-\epsilon}, 
\end{equation}
where the constant $B >0$ 
depends at most on $p$ and $\epsilon$.
\end{theo}

This theorem is our main technical result, 
from which  Theorem \ref{theo.BS.Leg} will follow (see Section \ref{sec.conclu}). 

\subsection{Proof of Theorem \ref{theo.spval.LBnd}}
 Let us start with the expression of $L^\ast(E_d/K, 1)$ obtained in Proposition \ref{lemm.expr.spval}: with the notations introduced there, one has:
\[ \log |L^\ast(E_d/K, 1)|  =   
\sum_{\Mm\in\O_q(V_d)} \log|\Mm| +  \sum_{\Mm\in\O_q(S_d)} \log\left|1-\frac{\Ja(\Mm)^2}{q^{|\Mm|}}\right|.\]
Although the first term here is positive, we know by Lemma \ref{lemm.prel}\eqref{item.iii.actq} that it is $o(d)$ when $d\to\infty$: 
consequently, proving Theorem \ref{theo.spval.LBnd} requires that we control how negative the second sum can be. 

Since $L^\ast(E_d/K, 1)$ is a rational number, the product $\pi_d^\ast :=\prod_{\Mm\in\O_q(S_d)} \left(1-\frac{\Ja(\Mm)^2}{q^{|\Mm|}}\right)$, \emph{a priori} an element of the cyclotomic field $K:=\Q(\zeta_d)$, is also rational. In particular, one has $\norm_{K/\Q}(\pi_d^\ast) = (\pi_d^\ast)^{[K:\Q]}$ and, by multiplicativity of the norm,
we obtain
\begin{align}\label{eq.prel.lower1}
\log |L^\ast(E_d/K, 1)| 
\geq  \log|\pi_d^\ast| 
&= \frac{\log\norm_{K/\Q}( \pi_d^\ast)}{[K:\Q]}  =  \frac{1}{[K:\Q]} \sum_{\Mm\in\O_q(S_d)} \log\norm_{K/\Q}\left(1-\frac{\Ja(\Mm)^2}{q^{|\Mm|}}\right) \notag\\ 
&=  \frac{1}{\phi(d)} \sum_{m\in S_d}\frac{1}{|\Mm|} \cdot \log\norm_{K/\Q}\left(1-\frac{\Ja(m)^2}{q^{|\Mm|}}\right).
\end{align}
Indeed, 
the value of $\Ja(\Mm)$ does not depend on the representative $m\in\Mm$ of that orbit (see Section \ref{sec.jacobi}).

We now try to obtain a more tractable expression of the right-hand side of \eqref{eq.prel.lower1}. Let us  first make use of the following lemma (inspired by \cite[Prop. 2.1]{Shioda_jacobi}): 
\begin{splemm}\label{lemm.norm.jacobi}
 Let $d\geq 2$ be an integer prime to $q$. For $m\in Z_d$, either $\Ja(m)^2= q^{|\Mm|}$, or 
\begin{equation}\label{eq.norm.jacobi}
\log\norm_{K/\Q}\left(1-\frac{\Ja(m)^2}{q^{|\Mm|}}\right)
\geq - (\log q^{|\Mm|}) \sum_{g\in(\Z/d\Z)^\times} \max\left\{ 0, 1- \frac{2\cdot \ord_{\gp} \Ja(g\cdot m)}{ \ord_p (q^{|\Mm|}) }\right\},
\end{equation}
where $\gp = \gP\cap \Q(\zeta_d)$ is the prime ideal of $K=\Q(\zeta_d)$ which lies below $\gP\subset\bar{\Z}$ (see Section \ref{sec.jacobi}), and 
where $\ord_p(.)$ denotes the $p$-adic valuation on $\Z$.
\end{splemm}

To avoid interrupting our current computation, we postpone the proof of this Lemma until the next subsection. 
Plugging \eqref{eq.norm.jacobi} in \eqref{eq.prel.lower1}, rearranging terms and dividing throughout by $\log q^d$ leads to:
\begin{align}\label{eq.prel.lower2}
\frac{\log |L^\ast(E_d/K, 1)|}{d\cdot\log q}
&\geq - \frac{1}{d}\sum_{m\in S_d} \left(\frac{1}{\phi(d)} \sum_{g\in(\Z/d\Z)^\times} \max\left\{ 0, 1- \frac{2\cdot \ord_{\gp} \Ja(g\cdot m)}{ \ord_p (q^{|\Mm|}) }\right\}\right) \notag \\
&\geq - \frac{1}{d}\sum_{m\in Z_d} \left(\frac{1}{\phi(d)} \sum_{g\in(\Z/d\Z)^\times} \max\left\{ 0, 1- \frac{2\cdot \ord_{\gp} \Ja(g\cdot m)}{ \ord_p (q^{|g\cdot\Mm|}) }\right\}\right),
\end{align}
because the terms we added are nonnegative, and because 
$|\Mm|=|g\cdot \Mm|$ for $g\in(\Z/d\Z)^\times$. 
To go further, we use the following variation on Stickelberger's theorem: 
\begin{splemm}[Stickelberger]\label{lemm.stickelberger}
 Let $d\geq 2$ be an integer prime to $q$, and $\gp$ be as above. For all $n\in Z_d$, the $\gp$-adic valuation of 
$\Ja(n)$ is given by  
\begin{equation}\label{eq.stickelberger}
\frac{\ord_{\gp} \Ja(n) }{\ord_p (q^{|\Mn|}) }
= \frac{1}{|\sbgp{p}{d}|} \cdot  \sum_{\pi\in\sbgp{p}{d}} 
\IND\left(\partfrac{\frac{\pi n}{d}}\right)
\end{equation}
 where $\sbgp{p}{d}\subset (\Z/d\Z)^\times$ is the subgroup generated by $p$, $\partfrac{.}$  denotes the fractional part,
and  $\IND:[0,1]\to\R$ is the characteristic function  of the interval $(0,1/2]$. 
\end{splemm}

The proof of this Lemma will also be given in the next subsection. For now, we use the result with $n=g\cdot m$
and 
rewrite, for all $m\in Z_d$:
\begin{align*}
\sum_{g\in(\Z/d\Z)^\times} \max\left\{ 0, 1- \frac{2\cdot \ord_{\gp} \Ja(g\cdot m)}{ \ord_p (q^{|g\cdot\Mm|}) }\right\}
&= 2\sum_{g\in(\Z/d\Z)^\times} \max\left\{ 0, \frac{1}{2}- \frac{1}{|\sbgp{p}{d}|}\sum_{\pi\in\sbgp{p}{d}} \IND\left(\partfrac{\frac{\pi gm}{d}}\right)\right\},
\end{align*}
where $1/2 = \int_{[0,1]}\IND$. 
Summing these identities over all $m\in Z_d$, we rewrite  inequality \eqref{eq.prel.lower2} under the following form:
\begin{equation}\label{eq.prel.lower3}
\frac{\log L^\ast(E_d/K, 1)}{d\cdot \log q} \geq -2\cdot \frac{1}{d}\sum_{m\in Z_d} E_p(m, d),\end{equation}
where $E_p(m,d)\geq 0$ is defined by
\[ E_p(m,d) = \frac{1}{\phi(d)} \sum_{g\in (\Z/d\Z)^\times} \max\left\{0,  \int_{[0,1]}\! \IND -\frac{1}{|\sbgp{p}{d}|}\sum_{\pi\in\sbgp{p}{d}} \IND\left(\partfrac{\frac{\pi gm}{d}}\right)  \right\}.\]
The proof of Theorem \ref{theo.spval.LBnd} is now reduced to 
showing that  
$\frac{1}{d}\sum_{m\in Z_d}E_p(m,d)$ tends to $0$  when $d\to\infty$. 
Since $\IND(x)\geq 0$, $E_p(m,d)$ satisfies $E_p(m,d)\leq 1/2$. 
For most $m\in Z_d$ though, a tighter upper bound holds (the proof of which will be given in subsection \ref{sec.proof.anal}):
\begin{splemm}\label{lemm.BND.good}
 Let $d\geq 2$ be an integer, coprime to $q$. For $m\in Z_d$, set $d_m=d/\gcd(m,d)$. 
 For all $\epsilon\in(0,1/4)$, one has 
\begin{equation}\label{eq.BND.good} 
E_p(m,d) \ll_{p,\epsilon} \left(\frac{\log\log d_m}{\log d_m}\right)^{1/4-\epsilon}, \end{equation}
where the implicit constant 
is effective and depends only on $p$ and $\epsilon$.
\end{splemm}  
\newcommand{\errou}{\Psi_{\epsilon}}

As suggested by \eqref{eq.BND.good}, we group the terms $m\in Z_d$ of the sum in \eqref{eq.prel.lower3} according to the value of $d_m=d/\gcd(d,m)$:
\[\sum_{m\in Z_d}E_p(m,d) = \sum_{\substack{e\mid d \\ e<d}} \sum_{\substack{m\in Z_d \\ d_m =e }} E_p(m,d).
\]
For each divisor $e$ of $d$, note that the set $\left\{m\in Z_d :  d_m=e\right\}$ contains exactly $|(\Z/e\Z)^\times|=\phi(e)$ elements.
 Since the bound \eqref{eq.BND.good} is good only when $d_m$ is large enough, we proceed to cut the last displayed sum 
 into two parts, with a parameter $u\in(0,1/2)$. 
 On the one hand, using the trivial bound $E_p(m,d)\leq 1/2$, we obtain that
\[ \sum_{\substack{e\mid d \\ e<d^u}} \sum_{\substack{m\in Z_d \\ d_m =e }} E_p(m,d) 
\leq \sum_{\substack{e\mid d \\ e<d^u}}  \frac{1}{2} \cdot\left|\left\{m\in Z_d :  d_m=e\right\}\right|
\leq \sum_{\substack{e\mid d \\ e<d^u}}\frac{\phi(e)}{2} 
 \leq  \sum_{1\leq e\leq d^u} \frac{e}{2} \ll d^{2u}.\] 
 On the other hand, using the refined bound \eqref{eq.BND.good} and the fact that the map $\errou : x\mapsto (\log\log x /\log x)^{1/4-\epsilon}$ is decreasing, we get that
\[\sum_{\substack{e\mid d \\ e\geq d^u}} \sum_{\substack{m\in Z_d \\ d_m =e }} E_p(m,d)
\ll_{p, \epsilon} \sum_{\substack{e\mid d \\ e\geq d^u}} \phi(e) \cdot \errou(e) 
\ll_{p, \epsilon}\errou(d^u) 
 \cdot \sum_{e\mid d}\phi(e)  
\ll_{p, \epsilon}\errou(d^u) \cdot d
\ll_{p, \epsilon} \frac{\errou(d)}{u^{1/4-\epsilon}}  \cdot d,\]
where the last inequality follows from 
$\frac{\log\log d^u}{\log d^u}\leq \frac{1}{u}\cdot \frac{\log\log d}{\log d}$.
Adding the two contributions, we deduce that 
 \[\frac{1}{d}\sum_{m\in Z_d} E_p(m,d)  
 \ll_{p, \epsilon}  d^{2u-1} 
 + \frac{1}{u^{1/4-\epsilon}}\cdot \left(\frac{\log\log d}{\log d}\right)^{1/4-\epsilon} 
 \ll_{p, \epsilon, u}  \left(\frac{\log\log d}{\log d}\right)^{1/4-\epsilon}.\] 
 
 Upon choosing a value $u\in(0,1/2)$ 
 and plugging this bound in the right-hand side of \eqref{eq.prel.lower3}, we arrive~at 
 \begin{equation}
\frac{\log L^\ast(E_d/K, 1)}{d\cdot \log q} \geq - B_0\cdot \left(\frac{\log\log d}{\log d}\right)^{1/4-\epsilon}, \notag\end{equation}
from which it readily follows that 
 \begin{equation}
\frac{\log L^\ast(E_d/K, 1)}{\log H(E_d/K)} \geq - B\cdot \left(\frac{\log\log d}{\log d}\right)^{1/4-\epsilon}, \notag\end{equation}
for some effective constant $B>0$ depending at most on $p, \epsilon$. 
Modulo the proofs of the three Lemmas \ref{lemm.norm.jacobi}, \ref{lemm.stickelberger} and \ref{lemm.BND.good}, this last inequality concludes the proof of Theorem \ref{theo.spval.LBnd}. \ProofEnd

\subsection{Proof of the algebraic Lemmas}

Let $d\geq 2$ be an integer coprime to $q$. As above, let $K=\Q(\zeta_d)$ be the $d$-th cyclotomic field, and $\gp$ be the prime ideal of $K$ which lies below the ideal $\gP\subset \bar{\Z}$ chosen in section \ref{sec.jacobi} (thus $\gp$ lies above $p$). 
We identify the Galois group $\Gal(K/\Q)$ with $(\Z/d\Z)^\times$ in the usual manner: to $t\in(\Z/d\Z)^\times$ corresponds $\sigma_t\in \Gal(K/\Q)$ defined by $\zeta_d\mapsto\zeta_d^t$.
We rely on well-known facts on the arithmetic of cyclotomic fields, for which the reader can consult \cite[Chap. 13]{IR}.

\begin{proof}[of Lemma \ref{lemm.norm.jacobi}]
Fix representatives $g_1=1, g_2, \dots, g_s\in(\Z/d\Z)^\times$ of the quotient $(\Z/d\Z)^\times/\sbgp{p}{d}$ of $(\Z/d\Z)^\times$  by the subgroup $\sbgp{p}{d}$ generated by $p$. 
For $i\in\{1, \dots, s\}$, put $\gp_i := (\sigma_{g_i})^{-1}\gp$,
so that $p$ decomposes in $K$ as the product $p\cdot \Z[\zeta_d] = \gp_1\gp_2\cdots \gp_s$.
 We note that $\norm\gp_i = \norm \gp = p^{\phi(d)/s}$ for all $i$. 

As in the statement of the Lemma, let $m\in Z_d$ such that $\Ja(m)^2\neq q^{|\Mm|}$, and define $v_m :=\ord_p(q^{|\Mm|})$. 
Since $p$ is unramified in $K$, $v_m =\ord_\gp(q^{|\Mm|})$ and it is clear that $q^{|\Mm|}\cdot \Z[\zeta_d]= \prod_{i=1}^s\gp_i^{v_m}$. The integral ideal generated by the Jacobi sum $\Ja(m)\in \Z[\zeta_d]$ is concentrated above $p$ (because $|\Ja(m)| = q^{|\Mm|/2} = p^{v_m/2}$): its decomposition as a product of prime ideals is $\Ja(m) = \prod_{i=1}^s \gp_i^{\ord_{\gp_i}\Ja(m)}$. 
It can be seen that the action of $\Gal(K/\Q)$ on $\{\Ja(n)\}_{n\in Z_d}$ is given by $\sigma_g(\Ja(n)) = \Ja(g\cdot n)$ for all $g\in(\Z/d\Z)^\times$. This gives that
\[\ord_{\gp_i}\Ja(m) = \ord_{\sigma_{g_i}^{-1}\gp}\Ja(m) = \ord_{\gp}\sigma_{g_i}\Ja(m) = \ord_{\gp} \Ja(g_i\cdot m).\]
Now, consider the ideal\newcommand{\ideal}{\mathcal{I}}
\[\ideal_{m} :=\prod_{i=1}^s\gp_i^{ \min\{v_m, 2\ord_{\gp}\Ja(g_i\cdot m)\} } = \prod_{g\in (\Z/d\Z)^\times/\sbgp{p}{d}} (\sigma_{g}^{-1} \gp)^{ \min\{v_m, 2\ord_{\gp}\Ja(g\cdot m)\} }.\]
By construction, $\ideal_m$ is an integral ideal in $K$, which divides the (nonzero) ideal generated by $(q^{|\Mm|} - \Ja(m)^2)$ in $\Z[\zeta_d]$. 
In particular, its norm $\norm \ideal_m$ divides $\norm_{K/\Q}(q^{|\Mm|} - \Ja(m)^2)$ in $\Z$. We infer that 
\[\norm_{K/\Q}\left(1-\frac{\Ja(m)^2}{q^{|\Mm|}}\right)  
= \frac{\norm_{K/\Q}(q^{|\Mm|} - \Ja(m)^2)}{\norm_{K/\Q}(q^{|\Mm|} )} 
\geq \frac{\norm\ideal_m}{\norm_{K/\Q}(q^{|\Mm|} )} = \frac{1}{q^{|\Mm| \cdot \phi(d)}\cdot (\norm\ideal_m)^{-1}}.\]
A straightforward computation from the definition of $\ideal_m$ implies that 
\[q^{|\Mm| \cdot \phi(d)}\cdot (\norm\ideal_m)^{-1} 
= q^{|\Mm| \cdot \sum_{g\in (\Z/d\Z)^\times} \max\left\{ 0, 1-\frac{2\ord_\gp\Ja(g\cdot m)}{v_m} \right\} }.\]
This uses our choice of $g_i$'s as representatives of $(\Z/d\Z)^\times/\sbgp{p}{d}$, and the fact that $\Ja(p^j\cdot m) = \Ja(m)$ for all $j\geq 0$.
Finally, from the last two displayed relations, we deduce that
\[\log \norm_{K/\Q}\left(1-\frac{\Ja(m)^2}{q^{\Mm}}\right)  
\geq -\log (q^{|\Mm|}) \cdot \sum_{g\in (\Z/d\Z)^\times} \max\left\{ 0, 1-\frac{2\ord_\gp\Ja(g\cdot m)}{v_m} \right\}, 
\]
as was to be proved. \ProofEnd\end{proof}

\begin{proof}[of Lemma \ref{lemm.stickelberger}]
Set $Q=q^{o_q(d)}$, $v= [\F_Q:\F_p] = \ord_p Q$
and $q'=q^{|\Mn|}$. 
The proof of Stickelberger's theorem
 gives the $\gp$-adic valuations of Jacobi sums (as in \cite[Chap. 14]{IR} for example, see also \cite[\S 4]{Ulmer_legII}). 
The result of that computation is that the Jacobi sum $\Ja(n)$ has $\gp$-adic valuation:
\[ \ord_\gp\Ja(n) = \frac{1}{[\F_Q:\F_{q'}]} \sum_{j=0}^{v-1} \left( -1+ 2 \partfrac{\frac{-np^j}{d}} + \partfrac{\frac{2np^j}{d}}\right). \]
One can check that $y\in[0,1] \mapsto -1 +2\partfrac{-y} + \partfrac{2y}$
is the characteristic function $\IND:[0,1]\to\R$ of the interval $(0,1/2]$, so that
 \begin{equation}\label{eq.stick.raw}
\ord_\gp \Ja(n) 
= \frac{1}{[\F_Q:\F_{q'}]} \sum_{j=0}^{v-1} 
\IND\left(\partfrac{\frac{np^j}{d}}\right).\end{equation}

There are repetitions in the sum over $j$: indeed,
since $q$ is a power of $p$, one has $v=\mathrm{lcm}([\F_q:\F_p], o_p(d))$ and thus, $v$ is a multiple of $o_p(d)$. 
By construction, $d$ divides $p^{o_p(d)}-1$ and any multiple thereof: it follows that 
we may reindex the sum over $j\in\intent{0,v-1}$ into a sum over $\pi\in\sbgp{p}{d}$ and obtain
\begin{equation}\label{eq.stick1}
\sum_{j=0}^{v-1} 
\IND\left(\partfrac{\frac{np^j}{d}}\right)
= \frac{v}{o_p(d)} \cdot \sum_{\pi\in\sbgp{p}{d}}
\IND\left(\partfrac{\frac{n\pi}{d}}\right).\end{equation}
Secondly, we note that
\begin{equation}\label{eq.stick2}
\frac{v}{o_p(d) \cdot [\F_Q:\F_{q'}]}  = \frac{[\F_Q:\F_p]}{o_p(d) \cdot [\F_Q:\F_{q'}]} = \frac{[\F_{q'}:\F_p]}{o_p(d)} = \frac{\ord_p(q^{|\Mn|})}{|\sbgp{p}{d}|}.
\end{equation}
Combining \eqref{eq.stick1} and \eqref{eq.stick2}  with \eqref{eq.stick.raw} yields the desired expression of $\ord_{\gp}\Ja(n)$.
\ProofEnd \end{proof}

\subsection{Proof of the analytic Lemma}\label{sec.proof.anal}

Before starting the proof, let us recall the following equidistribution statement:
\begin{theo} \label{theo.equidis}
Let $F:[0,1]\to\R$ be a function of bounded total variation, and denote by $\mathscr{V}(F)$ the total variation of $F$.
For an integer $d'\geq 2$,  suppose we are given an element $n \in(\Z/d'\Z)^\times$ and a subset $H$ of $(\Z/d'\Z)^\times$. 
Then, for all $\epsilon\in(0, 1/4)$, 
one has 
\begin{equation}\label{eq.equidis}
\frac{1}{\phi(d')} \sum_{g\in(\Z/d'\Z)^\times} 
\left| \int_{0}^1 F(t) \dd t  - \frac{1}{|H|} \sum_{h\in H} F\left(\partfrac{\frac{hgn }{d'}}\right) \right|  
\ll_{\epsilon} \mathscr{V}(F) \cdot \left(\frac{\log\log d'}{|H|}\right)^{1/4 - \epsilon}.\end{equation}
\end{theo}
We refer to \cite[Theorem 4.1]{Griffon_FermatS} for the proof of this theorem, and detailed comments. 

\medskip

Let $d\geq 2$ be an integer coprime to $q$, and $m\in Z_d$, we put $m':=m/\gcd(d,m)$ and $d':= d_m =d/\gcd(d,m)$. As in \eqref{eq.prel.lower3}, we set
\[ E_p(m,d) = \frac{1}{\phi(d)} \sum_{g\in (\Z/d\Z)^\times} \max\left\{0,  \int_{[0,1]}\! \IND -\frac{1}{|\sbgp{p}{d}|}\sum_{\pi\in\sbgp{p}{d}} \IND\left(\partfrac{\frac{\pi gm}{d}}\right)  \right\}.\]

\begin{proof}[of Lemma \ref{lemm.BND.good}] 
First, we observe that $E_p(m,d) = E_p(m', d')$. Indeed, the subgroup $\sbgp{p}{d'} \subset (\Z/d'\Z)^\times$ is the image of $\sbgp{p}{d}$ under the natural surjective morphism $(\Z/d\Z)^\times \to(\Z/d'\Z)^\times$, and this leads to 
\[\forall g\in(\Z/d\Z)^\times, \quad \frac{1}{|\sbgp{p}{d}|}\sum_{\pi\in\sbgp{p}{d}} \IND\left(\partfrac{\frac{\pi gm}{d}}\right) 
= \frac{1}{|\sbgp{p}{d}|}\sum_{\pi\in\sbgp{p}{d}} \IND\left(\partfrac{\frac{\pi gm'}{d'}}\right) 
=\frac{1}{|\sbgp{p}{d'}|}\sum_{\pi'\in\sbgp{p}{d'}} \IND\left(\partfrac{\frac{\pi' gm'}{d'}}\right).\]
A similar argument replaces the outer average in $E_p(m,d)$ (over $(\Z/d\Z)^\times$) by an average over $(\Z/d'\Z)^\times$, thus proving the claim. The upshot of this manipulation is that $\gcd(m',d')=1$, and we are now in a position to use Theorem \ref{theo.equidis}. 

Precisely, we apply Theorem \ref{theo.equidis}  to the step function $F=\IND$ with 
$n = m'$ and $H= \sbgp{p}{d'}$. 
Note that $|\sbgp{p}{d'}| \geq {\log d'}/{\log p}$ because $d'$ divides $p^{o_p(d')}-1$ by definition of the multiplicative order $o_p(d')=|\sbgp{p}{d'}|$ of $p\bmod{d'}$. 
Since $\IND$ is a step function on $[0,1]$, it is of bounded total variation; moreover, $\IND$ has only one ``jump'' of height $1$, so its total variation is $\mathscr{V}(\IND) =1$. 

Noticing that $\max\{0, y\}\leq |y|$ for all $y\in\R$, 
inequality \eqref{eq.equidis} here reads:
\begin{align*}
0\leq E_p(m,d) 
= E_p(m', d')
&\leq \frac{1}{\phi(d')} \sum_{g\in (\Z/d'\Z)^\times} \left|\int_{[0,1]}\IND -\frac{1}{|\sbgp{p}{d'}|}\sum_{\pi\in\sbgp{p}{d'}} \IND\left(\partfrac{\frac{\pi gm'}{d'}}\right) \right| \\
&\ll_\epsilon \left(\frac{\log\log d'}{|H|}\right)^{1/4 - \epsilon}
\ll_{p,\epsilon}  \left(\frac{\log\log d'}{\log d'}\right)^{1/4 - \epsilon}. 
\end{align*}
This concludes the proof.\ProofEnd\end{proof}

\section{Conclusion}\label{sec.conclu}

Finally regrouping the results of Theorems \ref{theo.spval.UBnd} and \ref{theo.spval.LBnd}, we obtain
  
 \begin{coro}\label{coro.spval.Bnd} Let $\F_q$ be a finite field of odd characteristic $p$ and $K=\F_q(t)$. For all $\epsilon\in(0,1/4)$, there are positive constants $A,B$ (depending at most on $p$ and $\epsilon$) such that: for all integer $d\geq 2$ prime to $q$, the special value $L^\ast(E_d/K, 1)$ satisfies: 
\begin{equation}\label{eq.spval.LBnd.ccl}
- B \cdot   \left(\frac{\log \log d}{\log d}\right)^{1/4-\epsilon}
\leq \frac{\log L^\ast(E_d/K, 1)}{\log H(E_d/K)} 
\leq A  \cdot  \frac{\log \log d}{\log d}.
\end{equation}

 \end{coro}

\begin{rema} By keeping track of constants in the estimates, one can make $A$ and $B$ explicit: it appears that 
$A = 48$ and ${B = 2  \left(32 + 4 (3\pi)^{-3} \epsilon^{-2}\right)\left(4\log p\right)^{1/4-\epsilon}}$
are suitable choices in \eqref{eq.spval.LBnd.ccl}.
\end{rema}

Together with Corollary \ref{coro.rel.spval.BS}, 
this Corollary \ref{coro.spval.Bnd} implies that
\[- B \cdot   \left(\frac{\log \log d}{\log d}\right)^{1/4-\epsilon} \leq \frac{\log\big( |\sha(E_d/K)|\cdot \Reg(E_d/K)\big)}{ \log H(E_d/K)} - 1 + O_q\left(\frac{1}{\log d}\right)
\leq A  \cdot  \frac{\log \log d}{\log d}.\]
In other words, we have proved the following quantitative version of Theorem \ref{theo.BS.Leg}: 
\begin{theo}\label{theo.BS.Leg.strong}
 Let $\F_q$ be a finite field of odd characteristic, and $K=\F_q(t)$. For any integer $d$ coprime with $q$, consider the Legendre elliptic curve $E_d/K$ as defined by \eqref{eq.Wmod}. Then, for all $\epsilon\in(0,1/4)$,  
  \[ \frac{\log\big( |\sha(E_d/K)|\cdot \Reg(E_d/K)\big)}{ \log H(E_d/K)} = 1 + O_{p, \epsilon}\left(\left(\frac{\log \log d}{\log d}\right)^{1/4-\epsilon}\right) \quad (\text{as }d\to\infty),
\]
where the implicit constant is effective and depends at most on $q, p$ and $\epsilon$.
\end{theo}

 \begin{rema} 
Throughout, we have imposed that $d$ be coprime to $q$, but 
Theorem \ref{theo.BS.Leg} 
sometimes holds without this assumption. 
For example, 
when $d=q^a$ (with an integer $a\geq 1$), we can prove the following. 

By Tate's algorithm, we compute  
$\deg\Delta_{\min}(E_{q^a}/K)= 6(q^a+1)$ and $\deg \cond(E_{q^a}/K) = 4$.
 The Grothendieck-Raynaud formula then implies that the $L$-function of $E_{q^a}$ 
must have degree 
$0$ as a polynomial in $T$.
Thus $L(E_{q^a}/K, T)=1\in\Z[T]$ and, in particular, one has $L^\ast(E_{q^a}/K, 1)=1$.  So that, for all $a\geq 1$, 
the estimates of Corollary \ref{coro.spval.Bnd} 
triviallly hold. 

By \cite[Remark 12.2]{Ulmer_legI}, the BSD conjecture for $E_d/K$ is actually true  for all integers $d\geq 2$. Upon using a suitable adaptation of Corollary \ref{coro.rel.spval.BS}, we deduce that 
  \[ \frac{\log\left( |\sha(E_{q^a}/K)|\cdot \Reg(E_{q^a}/K)\right)}{ \log H(E_{q^a}/K)} = 1 + O_{q}\left(\frac{\log q^a}{q^a}\right) \quad (\text{as }a\to\infty).
\]

In the ``intermediate cases'' where $d=q^a d'$ (with $a\geq 1$ and $d'$ coprime to $q$), 
\cite[Remark 12.2]{Ulmer_legI} shows that
$E_{d}$ and $E_{d'}$ are $K$-isogenous, and thus Theorem \ref{theo.BS.Leg} can be salvaged by invoking the invariance of its statement by isogeny (see \cite[\S8]{HP15}). 
\end{rema}


\newcommand{\mapolicebackref}[1]{%
         \hspace*{-5pt}{\textcolor{gray}{\small$\uparrow$ #1}}
}
\renewcommand*{\backref}[1]{
\mapolicebackref{#1}
}
\hypersetup{linkcolor=gray}

\pdfbookmark[0]{References}{references} 
\addcontentsline{toc}{section}{References}
\bibliographystyle{alpha}
\bibliography{../../Biblio_GENERAL.bib} 

\end{document}